\documentclass[12pt,reqno]{article}

\usepackage[usenames]{color}
\usepackage{amssymb}
\usepackage{amsmath}
\usepackage{amsthm}
\usepackage{amsfonts}
\usepackage{amscd}
\usepackage{graphicx}
\usepackage{tikz}
\usetikzlibrary{arrows.meta, automata, positioning}

\usepackage[colorlinks=true,
linkcolor=webgreen,
filecolor=webbrown,
citecolor=webgreen]{hyperref}

\definecolor{webgreen}{rgb}{0,.5,0}
\definecolor{webbrown}{rgb}{.6,0,0}

\usepackage{color}
\usepackage{fullpage}
\usepackage{float}

\usepackage{graphics}
\usepackage{latexsym}
\usepackage{epsf}
\usepackage{breakurl}

\newcommand{\seqnum}[1]{\href{https://oeis.org/#1}{\rm \underline{#1}}}

\def\Zee{\mathbb{Z}}
\def\Enn{\mathbb{N}}

\begin{document}

\theoremstyle{plain}
\newtheorem{theorem}{Theorem}
\newtheorem{corollary}[theorem]{Corollary}
\newtheorem{lemma}[theorem]{Lemma}
\newtheorem{proposition}[theorem]{Proposition}

\theoremstyle{definition}
\newtheorem{definition}[theorem]{Definition}
\newtheorem{example}[theorem]{Example}
\newtheorem{conjecture}[theorem]{Conjecture}

\theoremstyle{remark}
\newtheorem{remark}[theorem]{Remark}

\begin{center}
\vskip 1cm{\LARGE\bf A $2$-Regular Sequence That Counts The Divisors of $n^2 + 1$
}
\vskip 1cm
\large
Anton Shakov\\
Department of Mathematics and Statistics\\
Queen's University\\
48 University Avenue\\
Kingston, ON K7L 3N6\\
Canada\\
\href{mailto:anton.shakov@queensu.ca}{\tt anton.shakov@queensu.ca} \\
\end{center}

\vskip .2 in
\begin{abstract}
We introduce the $2$-regular integer sequence \seqnum{A383066} $= (s(n))_{n \geq 1}$,
which begins $0, 1, 1, 2, 3, 3, 2, \ldots$.
We prove that the number of occurrences of an integer $m \geq 0$ in this sequence is equal to $\tau(m^2+1)$, the number of divisors of $m^2 + 1$. Using this fact, we give a generating function for $\tau(m^2+1)$. We also discuss other interesting properties of $s(n)$, including its relationship to the Fibonacci sequence.
\end{abstract}

\section{Introduction and proof of the main result}

We begin by recalling the definition of $k$-regular sequences, which were introduced by Allouche and Shallit \cite{AS1} as a generalization of automatic sequences \cite{AS2}. 

\begin{definition}
A sequence \(s(n)\) is \(k\)-regular if there exists an integer \(E\) such that, for all \(e_j > E\) and \(0 \leq r_j \leq k^{e_j} - 1\), every subsequence of \(s\) of the form \(s(k^{e_j}n + r_j)\) is expressible as an $\Zee$-linear combination 
\[
\sum_{i}c_{ij}s(k^{f_{ij}}n+b_{ij}),
\]
where \(f_{ij} \leq E\), and \(0 \leq b_{ij} \leq k^{f_{ij}} - 1.\)
\end{definition}

In the previous definition, the integers $\Zee$ can be replaced by any commutative Noetherian ring $R'$, in which case we would say that $s(n)$ is $(R', k)$-regular. However, for the purposes of this paper, we consider only integer sequences. We begin by giving some well-known examples of $2$-regular integer sequences. 

\begin{example}
 The $2$-adic valuation of a positive integer $n$ \seqnum{A007814}, defined by $v_2(n) := \sup\{k \in \Enn_0: 2^k \mid n \}$ is a $2$-regular sequence, since it satisfies the recursions
\[
\begin{cases}
v_2(2k+1) = 0\\
v_2(2k) = v_2(k)+1 \\
\end{cases}
\]
with initial condition $v_2(1)=0$.
\end{example}

\begin{example}
The Cantor sequence \seqnum{A005823} is a $2$-regular sequence which consists of integers whose ternary expansions contain no $1$s. The Cantor sequence $c(n)$ satisfies the recursions 
\[
\begin{cases}
c(2k) = 3c(k)+2 \\
c(2k+1) = 3c(k+1) \\
\end{cases}
\]
with initial condition $c(1) = 0$. 
\end{example}

For more examples of $k$-regular sequences, see Allouche and Shallit \cite[pp.~186--194]{AS1}. We now state the main result of this paper.

\begin{theorem}\label{main}
We have
$$\sum_{m\geq0} \tau(m^2+1)x^m = \sum_{n\geq1} x^{s(n)}$$
where $\tau$ is the usual divisor counting function and $s(n)_{n\geq1}$ is a $2$-regular sequence defined recursively by
\[
\begin{cases}
s(4k) = 2s(2k) - s(k)\\
s(4k+1) = 2s(2k) + s(2k + 1) \\
s(4k+2) = 2s(2k + 1) + s(2k) \\
s(4k+3) = 2s(2k + 1) - s(k) \\
\end{cases}
\]
with initial conditions $s(1) = 0,\ s(2)=1,\ s(3) = 1$.
\end{theorem}

In other words, we prove that 
$$\#\{n: s(n) = m \} = \tau(m^2 + 1)$$ 
for all integers $m \geq 0$.

\begin{proof}
Consider the binary tree of integer pairs $(d, m)$ generated in the following way. We begin with the pair $(1, 0)$. Each pair has two children, left and right, given by the maps $L(d, m) := (d, m+d)$ and $R(d, m) :=  \left( \frac{(m+d)^2+1}{d}, m + \frac{m^2+1}{d}\right)$. For the first four rows of the tree, see Figure~\ref{fig:main_tree}.
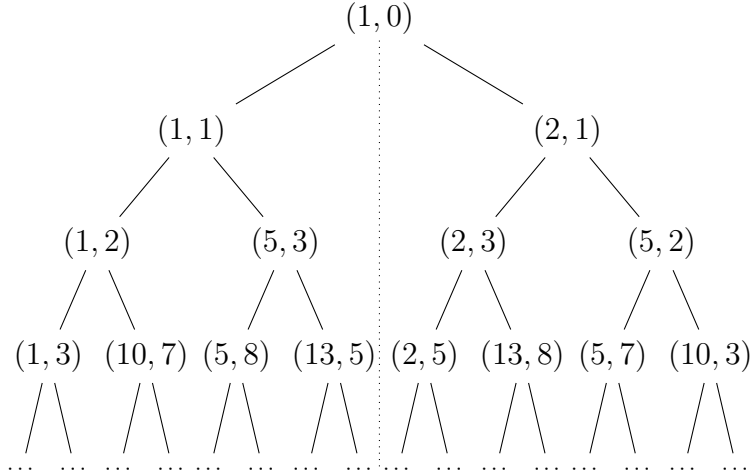
\begin{figure}[htb]
\centering
\begin{tikzpicture}[level distance=15mm, 
                    level 1/.style={sibling distance=50mm},
                    level 2/.style={sibling distance=25mm},
                    level 3/.style={sibling distance=13mm},
                    level 4/.style={sibling distance=7mm}]
  \draw[dotted] (0,-0.3) -- (0,-6);

  \node {\((1, 0)\)}
    child {
      node {\((1, 1)\)} 
      child {
        node {\((1, 2)\)}
        child {node {\((1, 3)\)}
          child {node {\scriptsize $\cdots$}}
          child {node {\scriptsize $\cdots$}}
        }
        child {node {\((10, 7)\)}
          child {node {\scriptsize $\cdots$}}
          child {node {\scriptsize $\cdots$}}
        }
      }
      child {
        node {\((5, 3)\)}
        child {node {\((5, 8)\)}
          child {node {\scriptsize $\cdots$}}
          child {node {\scriptsize $\cdots$}}
        }
        child {node {\((13, 5)\)}
          child {node {\scriptsize $\cdots$}}
          child {node {\scriptsize $\cdots$}}
        }
      }
    }
    child {
      node {\((2, 1)\)}
      child {
        node {\((2, 3)\)}
        child {node {\((2, 5)\)}
          child {node {\scriptsize $\cdots$}}
          child {node {\scriptsize $\cdots$}}
        }
        child {node {\((13, 8)\)}
          child {node {\scriptsize $\cdots$}}
          child {node {\scriptsize $\cdots$}}
        }
      }
      child {
        node {\((5, 2)\)}
        child {node {\((5, 7)\)}
          child {node {\scriptsize $\cdots$}}
          child {node {\scriptsize $\cdots$}}
        }
        child {node {\((10, 3)\)}
          child {node {\scriptsize $\cdots$}}
          child {node {\scriptsize $\cdots$}}
        }
      } 
    };
\end{tikzpicture}
\caption{Integer pair tree.}
\label{fig:main_tree}
\end{figure}

The involution map $\iota(d, m) := (\frac{m^2+1}{d}, m)$ sends each pair to its reflection with respect to the tree's central line of symmetry, represented by the dotted line in Figure~\ref{fig:main_tree}. We note that $R(d, m) = (\iota \circ L \circ \iota) (d, m)$, which can either be checked by direct computation or by fixing an integer pair $(d, m)$ on the tree and visually seeing that $\iota \circ L \circ \iota$ (reflection, left-child map, reflection) sends $(d, m)$ to the same pair as the right-child map $R(d, m)$.

\begin{lemma}\label{only_div_pairs}
If a pair $(d, m)$ appears on the integer pair tree then $d \geq 1$, $m \geq 0$, and $d$ divides $m^2+1$. 
\end{lemma}

\begin{proof}
Suppose an integer pair $(d, m)$ appears on the tree with the properties that $d \geq 1$, $m \geq 0$, and $d \mid (m^2+1)$. We claim that these properties also hold for the transformed pairs $L(d, m)$ and $R(d, m)$. It is easy to see that both transformed pairs $L(d,m)$ and $R(d, m)$ are still integer pairs, that their first components are $\geq 1$, and that their second components are $\geq 0$. To see that the first component still divides $1+$ the square of the second component, it suffices to check that this property is preserved by $L$ and $\iota$, since we saw that $R = \iota \circ L \circ \iota$. Indeed,
the property $d \mid (m^2+1)$ is preserved by both
$L$ and $\iota$, since 
$$d \mid (m^2 + 1) \implies d \mid \left((m+d)^2 + 1\right)
\text{ and } \frac{m^2+1}{d} \mid (m^2 + 1).$$
The first integer pair on the tree is $(1,0)$, which satisfies all three properties. Therefore, all of its descendants must also satisfy all three properties.
\end{proof}

\begin{lemma}\label{unique_lemma}
Suppose $d \geq 1$, $m \geq 0$, and $d$ divides $m^2+1$. Then the pair $(d, m)$ appears on the integer pair tree exactly once.
\end{lemma}

\begin{proof}
We begin by restating the lemma so that it can be proved using induction. We let $P(M)$ denote the statement ``If $d \geq 1$ divides $m^2 + 1$ with $0 \leq m \leq M$, then the pair $(d, m)$ appears on the tree exactly once."

To prove the lemma, we show that $P(M)$ is true for all integers $M \geq 0$ by induction on $M$. As our base case, we see that $P(0)$ is true, since the pair $(1,0)$ appears exactly once on the tree, in the first row. This is because $L$ and $R$ each increase the second component of a pair by at least $1$, so there are no more pairs on the tree with second component $0$. 

Now suppose $M > 0$. Our induction assumption is that $P(M-1)$ is true. Namely, we assume that for all $d \geq 1$ with $d \mid (m^2 + 1)$ and $0 \leq m \leq M-1$ we have $(d, m)$ appearing on the tree exactly once. We show that this implies $P(M)$ is true by assuming that some $d \geq 1$ divides $M^2 + 1$ and using the induction assumption to prove that $(d, M)$ must appear on the tree exactly once. The claim that $(d, M)$ appears on the tree exactly once is equivalent to the statement that there exists a unique path from the root pair $(1, 0)$ to $(d, M)$ in terms of the maps $L$ and $R$. It is easy to check that $L^{-1}(d, M) = (d, M-d)$ and $R^{-1}(d, M) =  \left( \frac{(M-d)^2+1}{d}, M - \frac{M^2+1}{d}\right)$. We show that exactly one of the second components of these inverse mappings $\{M-d, \ M -  \frac{M^2+1}{d}\}$ is nonnegative (which we showed in the previous lemma is a necessary condition for pairs to appear on the tree). The nonnegativity of exactly one of $\{M-d, \ M -  \frac{M^2+1}{d}\}$ is a result of the inequalities
\[\inf \left\{d, \frac{m^2+1}{d} \right\} \leq m < \sup\left\{d, \frac{m^2+1}{d}\right\},\]
which hold for all positive integers $d$ and $m$ with $d \mid (m^2+1)$. These inequalities can be proved by considering the cases $d \leq m$ and $d > m$ and using the fact that $m^2 < m^2 + 1 < (m+1)^2$ for all $m > 0$. Equality occurs in the left inequality when $(d,m) \in \{(1,1), (2,1) \}$. Thus, exactly one of the pairs $\{L^{-1}(d, M), R^{-1}(d, M) \}$ has a nonnegative second component. Furthermore, this component is strictly less than $M$. The properties $d \mid (m^2 + 1)$ and $d \geq 1$ are clearly preserved by $L^{-1}$ and $R^{-1}$. Therefore, by our induction assumption and Lemma~\ref{only_div_pairs}, exactly one of the pairs $\{L^{-1}(d, M), R^{-1}(d, M) \}$ appears on the tree and there exists a unique path from $(1, 0)$ to this pair in terms of $L$ and $R$. In other words, the pair $(d, M)$ has exactly one parent appearing on the tree, which is guaranteed by our induction assumption to have a unique path back to $(1, 0)$ in terms of $L$ and $R$. 

Finally, this proves that $(d, M)$ has a unique path to $(1,0)$ in terms of $L$ and $R$ and therefore appears on the tree exactly once, which shows that $P(M)$ is true.
\end{proof}
We now write only the second pair components as they appear on the integer pair tree. 
\begin{figure}[htb]
\centering
\begin{tikzpicture}[level distance=1.5cm,
  level 1/.style={sibling distance=4cm},
  level 2/.style={sibling distance=2cm},
  level 3/.style={sibling distance=1cm},
  level 4/.style={sibling distance=0.5cm}]
  \node {$0$}
    child {node {$1$}
      child {node {$2$}
        child {node {$3$}
          child {node {\scriptsize $\cdots$}}
          child {node {\scriptsize $\cdots$}}
        }
        child {node {$7$}
          child {node {\scriptsize $\cdots$}}
          child {node {\scriptsize $\cdots$}}
        }
      }
      child {node {$3$}
        child {node {$8$}
          child {node {\scriptsize $\cdots$}}
          child {node {\scriptsize $\cdots$}}
        }
        child {node {$5$}
          child {node {\scriptsize $\cdots$}}
          child {node {\scriptsize $\cdots$}}
        }
      }
    }
    child {node {$1$}
      child {node {$3$}
        child {node {$5$}
          child {node {\scriptsize $\cdots$}}
          child {node {\scriptsize $\cdots$}}
        }
        child {node {$8$}
          child {node {\scriptsize $\cdots$}}
          child {node {\scriptsize $\cdots$}}
        }
      }
      child {node {$2$}
        child {node {$7$}
          child {node {\scriptsize $\cdots$}}
          child {node {\scriptsize $\cdots$}}
        }
        child {node {$3$}
          child {node {\scriptsize $\cdots$}}
          child {node {\scriptsize $\cdots$}}
        }
      }
    };
\end{tikzpicture}
\caption{Second component tree.}
\label{fig:second_comp}
\end{figure}
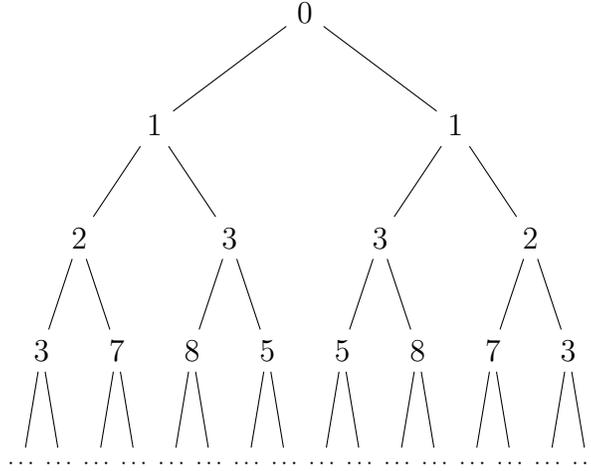 
Let us temporarily define $s(n)$ as the sequence one gets by reading the integers on the second component tree left-to-right or right-to-left, starting from the top. Figure~\ref{fig:second_comp} shows the first four rows of the second component tree. For example, $s(1) = 0,\ s(2) = 1,\ s(3) = 1$, etc. We will show that this agrees with the original definition we gave in Theorem~\ref{main}.  Using the new definition, and in light of Lemmas~\ref{only_div_pairs} and~\ref{unique_lemma}, this proves that the number of occurrences of an integer $m \geq 0$ on the second component tree is equal to $\#\{(d,m): d \geq 1,\ d \mid (m^2 + 1)\} = \tau(m^2+1)$.
To see that $s(n)$ satisfies the recursions we gave in Theorem~\ref{main}, we keep track of the second components as they are changed by the maps $L$ and $R$. For example, since $L(d, m) = (d, m+d)$ and $R(d, m) = \left( \frac{(m+d)^2+1}{d}, m + \frac{m^2+1}{d}\right)$, we write $m_L$ for $m+d$ and $m_R$ for $m + \frac{m^2+1}{d}$. Figure~\ref{fig:m_comps} shows three generations of second pair components. Figure~\ref{fig:m_relations} shows how to write components in the third generation as linear combinations of components from the previous two generations.

We now rewrite the parent-child relationships in terms of the sequence $s(n)$. Reading a two-child binary tree left-to-right, we see that for a parent with index $k$, its left child's index is $2k$ while its right child's index is $2k+1$. Figure~\ref{fig:tree-sk} relates three generations of components to the corresponding indices of $s(n)$.

Using the linear dependencies we found, we finally recover the recursions from Theorem~\ref{main}, namely
\[
\begin{cases}
s(4k) = 2s(2k) - s(k) \\
s(4k+1) = 2s(2k) + s(2k + 1)\\
s(4k+2) = 2s(2k + 1) + s(2k)  \\
s(4k+3) = 2s(2k + 1) - s(k). \\
\end{cases}
\]
\end{proof}

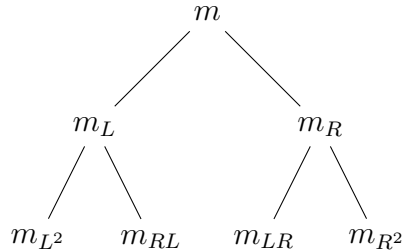
\begin{figure}[htb]
\centering
\begin{tikzpicture}[level distance=1.5cm,
  level 1/.style={sibling distance=3cm},
  level 2/.style={sibling distance=1.5cm}]
  \node {$m$}
    child {node {$m_L$}
      child {node {$m_{L^2}$}}
      child {node {$m_{RL}$}}
    }
    child {node {$m_R$}
      child {node {$m_{LR}$}}
      child {node {$m_{R^2}$}}
    };
\end{tikzpicture}
\caption{Three generations of second pair components.}
\label{fig:m_comps}
\end{figure}
\begin{figure}[htb]
\centering
\begin{tikzpicture}[level distance=2cm,
  level 1/.style={sibling distance=7.5cm},
  level 2/.style={sibling distance=3.5cm}]
  \node {$m$}
    child {node {$m_L$}
      child {node {$m_{L^2}=2m_L -  m$}}
      child {node {$m_{RL}=2m_L + m_R$}}
    }
    child {node {$m_R$}
      child {node {$m_{LR}=2m_R + m_L$}}
      child {node {$m_{R^2}=2m_R - m$}}
    };
\end{tikzpicture}
\caption{Linear dependencies between second pair components.}
\label{fig:m_relations}
\end{figure}
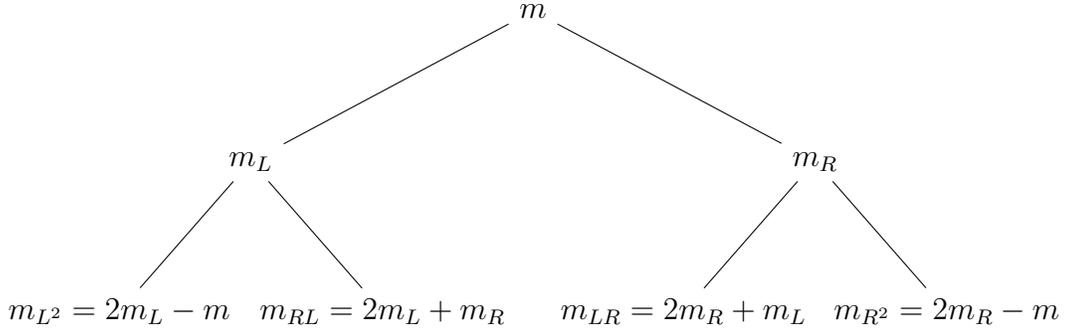

\begin{figure}[htb]
\centering
\begin{tikzpicture}[level distance=2cm,
  level 1/.style={sibling distance=7.5cm},
  level 2/.style={sibling distance=3.5cm}]
  \node {$m=s(k)$}
    child {node {$m_L=s(2k)$}
      child {node {$m_{L^2}=s(4k)$}}
      child {node {$m_{RL} =s(4k+1)$}}
    }
    child {node {$m_R=s(2k+1)$}
      child {node {$m_{LR} = s(4k+2)$}}
      child {node {$m_{R^2} = s(4k+3)$}}
    };
\end{tikzpicture}
\caption{Dependencies re-indexed in terms of $s(n)$.}
\label{fig:tree-sk}
\end{figure}
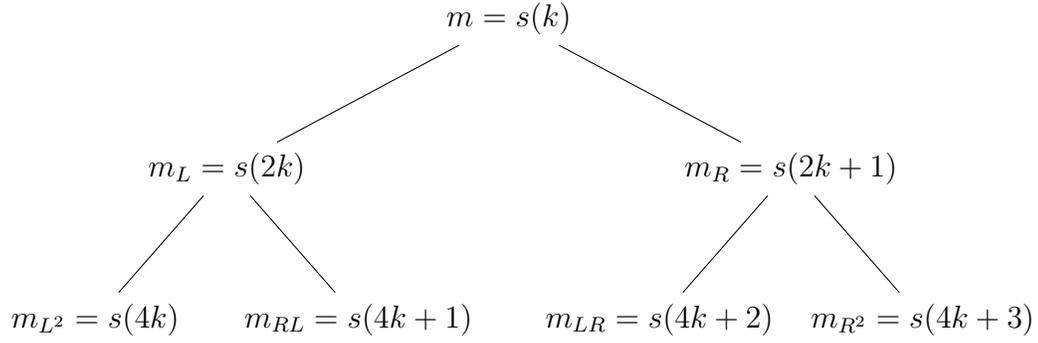

\section{Other properties of the sequence}
We discuss some other interesting properties of the sequence $s(n)$. From the recursions for $s(n)$, it is easy to show that the row sums of the second component tree satisfy the linear recurrence $r_n = 5r_{n-1} - 2r_{n-2}$ with initial conditions $r_0 = 0,\ r_1 = 2$. Here, $r_n$ denotes the sum of integers on row $n\geq0$ of the second component tree, or, equivalently, $r_n := \sum_{2^n \leq t < 2^{n+1}}s(t)$. By diagonalizing the $2 \times 2$ integer matrix corresponding to this recurrence, we write down an exact formula for the average value of an integer on row $n \geq 0$ of the tree.
\begin{proposition}
\[
\frac{1}{2^n} \sum_{2^n \leq t < 2^{n+1}} s(t) 
= \frac{(5 + \sqrt{17})^n - (5 - \sqrt{17})^n}{2^{2n - 1} \sqrt{17}}
\]
\end{proposition}

\begin{proposition}
The integer $n^2 + 1$ is a prime number if and only if 
$$\{m: s(m) = n\} = \{2^n, 2^{n+1}-1\}.$$
\end{proposition}

\begin{proof}
Note that $s(2^n) = s(2^{n+1} - 1) = n$ for all $n \geq 0$. These correspond to the leftmost and rightmost integers on row $n$ of the second component tree. In other words, $\{2^n, 2^{n+1}-1\} \subseteq \{m: s(m) = n \}$.  Now, if $n^2 + 1$ is a prime number, then $\tau(n^2 + 1) = 2$ and it follows from Theorem~\ref{main} that $\{m: s(m) = n\} = \{2^n, 2^{n+1}-1\}$.
Conversely, if $n^2 + 1$ is composite, then $\tau(n^2 + 1) > 2$ and so Theorem~\ref{main} implies there is some $m \notin \{2^n, 2^{n+1} - 1\}$ such that $s(m) = n$.
\end{proof}
\begin{figure}[htb]
  \centering
  \includegraphics[width=0.8\textwidth]{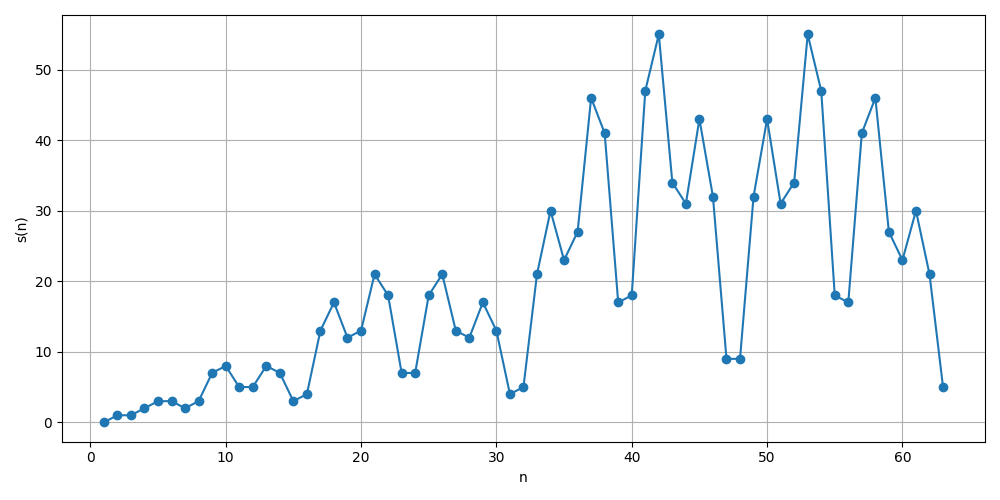}
  \caption{Line plot of the sequence \( s(n) \) for $n \in [1, 63]$}
  \label{fig:sequence_plot}
\end{figure}

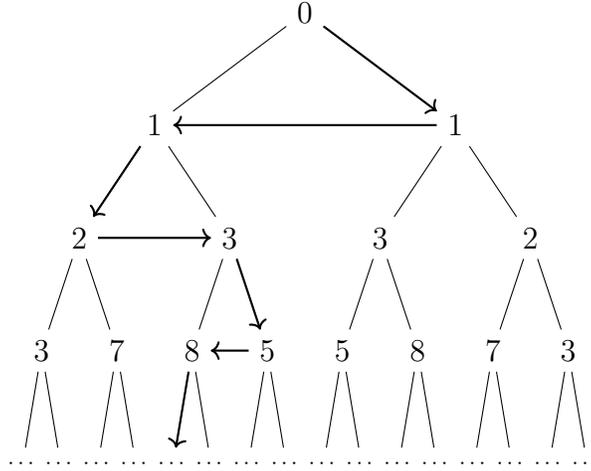
\begin{figure}[htb]\label{fib_path}
\centering
\begin{tikzpicture}[
  level distance=1.5cm,
  level 1/.style={sibling distance=4cm},
  level 2/.style={sibling distance=2cm},
  level 3/.style={sibling distance=1cm},
  level 4/.style={sibling distance=0.5cm},
  every node/.style={draw=none}
  ]
  \node (A) {$0$}
    child {node (B) {$1$}
      child {node (C) {$2$}
        child {node (D) {$3$}
          child {node (D1) {\scriptsize $\cdots$}}
          child {node (D2) {\scriptsize $\cdots$}}
        }
        child {node (E) {$7$}
          child {node {\scriptsize $\cdots$}}
          child {node {\scriptsize $\cdots$}}
        }
      }
      child {node (F) {$3$}
        child {node (G) {$8$}
          child {node (G1) {\scriptsize $\cdots$}}
          child {node {\scriptsize $\cdots$}}
        }
        child {node (H) {$5$}
          child {node {\scriptsize $\cdots$}}
          child {node {\scriptsize $\cdots$}}
        }
      }
    }
    child {node (I) {$1$}
      child {node (J) {$3$}
        child {node (J1) {$5$}
          child {node {\scriptsize $\cdots$}}
          child {node {\scriptsize $\cdots$}}
        }
        child {node (J2) {$8$}
          child {node {\scriptsize $\cdots$}}
          child {node {\scriptsize $\cdots$}}
        }
      }
      child {node (K) {$2$}
        child {node {$7$}
          child {node {\scriptsize $\cdots$}}
          child {node {\scriptsize $\cdots$}}
        }
        child {node {$3$}
          child {node {\scriptsize $\cdots$}}
          child {node {\scriptsize $\cdots$}}
        }
      }
    };

  \draw[->, thick] (A) -- (I);
  \draw[->, thick] (B) -- (C);
  \draw[->, thick] (C) -- (F);
  \draw[->, thick] (F) -- (H);
  \draw[->, thick] (I) -- (B);
  \draw[->, thick] (H) -- (G);
  \draw[->, thick] (G) -- (G1);

\end{tikzpicture}
\caption{Path inside the second component tree that runs over the Fibonacci sequence.}
\label{fig:Fibonacci_path}
\end{figure}

\begin{figure}[htb]
\centering
\begin{tikzpicture}[level distance=1.5cm,
  level 1/.style={sibling distance=6cm},
  level 2/.style={sibling distance=3cm},
  level 3/.style={sibling distance=1.5cm},
  every node/.style={draw=none}]
  
  \node (Fn) {$F_{n}$}
    child {node (Fn1) {$F_{n+1}$}
      child {node {\scriptsize $\cdots$}}
      child {node {\scriptsize $\cdots$}}
    }
    child {node (Fn2) {$F_{n+2}$}
      child {node (Fn4) {$F_{n+4}$}}
      child {node (Fn3) {$F_{n+3}$}}
    };

  \draw[->, thick] (Fn) -- (Fn1);
  \draw[->, thick] (Fn1) -- (Fn2);
  \draw[->, thick] (Fn2) -- (Fn3);
  \draw[->, thick] (Fn3) -- (Fn4);
\end{tikzpicture}
\caption{Fibonacci path in the case where $m_L = F_{n+1}$ and $m_R = F_{n+2}$.}
\label{fig:path1}
\end{figure}
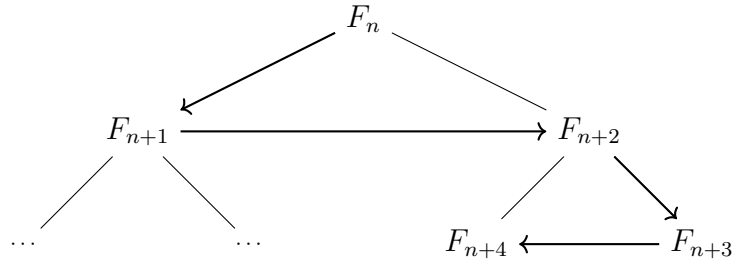

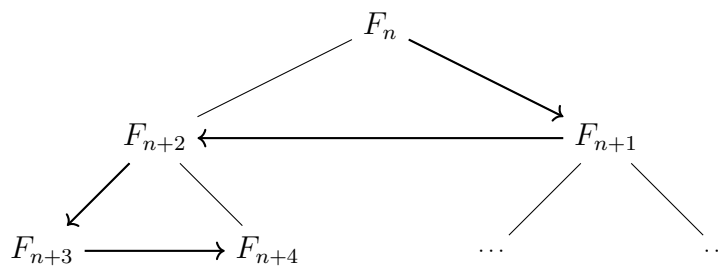
\begin{figure}[htb]
\centering
\begin{tikzpicture}[level distance=1.5cm,
  level 1/.style={sibling distance=6cm},
  level 2/.style={sibling distance=3cm},
  level 3/.style={sibling distance=1.5cm},
  every node/.style={draw=none}]
  
  \node (Fn) {$F_{n}$}
    child {node (Fn2) {$F_{n+2}$}
      child {node (Fn3) {$F_{n+3}$}}
      child {node (Fn4) {$F_{n+4}$}}
    }
    child {node (Fn1) {$F_{n+1}$}
      child {node {\scriptsize $\cdots$}}
      child {node {\scriptsize $\cdots$}}
    };

  \draw[->, thick] (Fn) -- (Fn1);
  \draw[->, thick] (Fn1) -- (Fn2);
  \draw[->, thick] (Fn2) -- (Fn3);
  \draw[->, thick] (Fn3) -- (Fn4);
\end{tikzpicture}
\caption{Fibonacci path in the case where $m_L = F_{n+2}$ and $m_R = F_{n+1}$.}
\label{fig:path2}
\end{figure}

The Fibonacci sequence also makes an appearance in the second component tree. The Fibonacci sequence is defined by the recursion $F_{n+1} = F_n + F_{n-1}$ with initial conditions $F_1 = 0, F_2 = 1$.
\begin{proposition}\label{fibonacci_prop}
Consider the sequence defined by
\[
a(n) = 
\begin{cases}
1, & \text{if } n = 1; \\
2a(n-1), & \text{if } n = 4k; \\
a(n-1) + 1, & \text{if } n = 4k+1; \\
2a(n-1) + 1, & \text{if } n = 4k+2; \\
a(n-1) - 1, & \text{if } n = 4k+3.
\end{cases}
\]
Then $s(a(n)) = F_n$, namely the $n$th term of the Fibonacci sequence.
\end{proposition}

\begin{proof}
Note that the recursions for $a(n)$ are chosen so that the sequence $s(a(n))$ sweeps out the path on the second component tree shown in Figure~\ref{fig:Fibonacci_path}. Starting from the top of the tree, move to the closest neighbor in a given direction, cycling through these four directions: south-east, west, south-west, east.

The recursions for $s(n)$ tell us that if we start with $m = F_n$ and $\{m_L, m_R\} = \{F_{n+1}, F_{n+2}\}$, the children of $F_{n+2}$ will be $2F_{n+2} + F_{n+1}$ and $2F_{n+2} - F_{n}$. Using the Fibonacci recursion we have that $2F_{n+2} - F_{n} = F_{n+2} + F_{n+1} = F_{n+3}$ and $2F_{n+2}+F_{n+1} = F_{n+2} + F_{n+3} = F_{n+4}$. Therefore, if $m_L = F_{n+1}$ and $m_R = F_{n+2}$ we trace out the path given in Figure~\ref{fig:path1}. If $m_L = F_{n+2}$ and $m_R = F_{n+1}$ we trace the path given in Figure~\ref{fig:path2}.

Since $F_n = s(a(n))$ for $1 \leq n \leq 3$, we conclude that $s(a(n)) = F_n$ for all $n \geq 1$.
\end{proof}

\begin{corollary}
The integer $F_n^2 + 1$ is composite for $n>4$.
\end{corollary}

\begin{proof}
This follows from the fact that there exists $m \notin \cup_{n\geq 0} \{2^n, 2^{n+1}-1\}$ with $s(m)=F_n$ for all $n > 4$, as can be seen in Figure~\ref{fig:Fibonacci_path}.
\end{proof}

Another way to see this fact is to use Cassini's Identity \cite{WZ}, namely $F_{n-1}F_{n+1} - F_n^2 = (-1)^n$ in the case where $n=2k$, as well as the related identity $F_{2k-1}F_{2k+3} - F_{2k+1}^2 = 1$ in the case where $n = 2k+1$. Both identities can be proved using induction.

\begin{corollary}
The largest integer on row $n\geq 1$ of the second component tree is $F_{2n}$.
\end{corollary}

\begin{proof}
It can be seen from the recursions for the second component tree that the largest second component on a particular row $n$ of the integer pair tree is given by either one of the zigzag paths $RLRLRL\cdots$ or $LRLRLR\cdots$. This, together with Proposition~\ref{fibonacci_prop}, proves the corollary. 
\end{proof}

\section{Acknowledgments} I would like to thank Dr.\ Brad Rodgers for his insight and patience in helping me organize this paper, and Christian Kudeba for our many productive conversations. I would also like to thank the anonymous referees for their helpful comments and suggestions.

\bigskip
\hrule
\bigskip

\noindent 2020 {\it Mathematics Subject Classification}:
Primary 11B37; Secondary 11A25, 05C05, 11B39.

\noindent \emph{Keywords:} $k$-regular sequence, divisor function, Fibonacci number, divisors of integer-valued polynomials.

\bigskip
\hrule
\bigskip

\noindent (Concerned with sequence
\seqnum{A383066}.)

\bigskip
\hrule
\bigskip

\end{document}